\documentclass[12pt]{amsart}

\usepackage{amsmath,amsthm,amsfonts,latexsym,amscd,amssymb,enumerate, tikz,color}
\input{xypic}

\usepackage{amssymb}

\def \<{\langle}
\def \>{\rangle}

\newcommand{\ve}{\mathbf{v}}
\newcommand{\ho}{\mathbf{h}}

\newcommand{\Z}{\mathbb{Z}}
\newcommand{\R}{\mathbb{R}}
\newcommand{\CP}{\mathbb{CP}}
\newcommand{\RP}{\mathbb{RP}}
\newcommand{\HP}{\mathbb{HP}}

\newcommand{\calV}{\mathcal{V}}
\newcommand{\calH}{\mathcal{H}}

\newcommand{\calF}{\mathcal{F}}

\newcommand{\tildeA}{\widetilde{A}}

\newcommand{\dual}{\calF^{\#}}

\newcommand{\Id}{Id}

\newcommand{\g}{\gamma}

\newcommand{\w}{\wedge}

\newcommand{\la}{\langle}
\newcommand{\ra}{\rangle}

\renewcommand{\prod}[2]{\la\,#1 , #2\,\ra}
\newcommand{\ip}[2]{\ensuremath{\langle #1,#2\rangle}}

\newcommand{\boundary}{\partial}

\newtheorem{lem}{Lemma}

\newtheorem{cor}{Corollary}
\newtheorem{defi}{Definition}
\newtheorem{thm}{Theorem}

\theoremstyle{definition}
\newtheorem{exa}{Example}

\begin{document}

\author{Pablo Angulo-Ardoy}
\address{ Department of Mathematics, Universidad Aut\'onoma de Madrid}
\curraddr{}
\email{pablo.angulo@uam.es}

\author{Luis Guijarro}
\address{ Department of Mathematics, Universidad Aut\'onoma de Madrid, and ICMAT CSIC-UAM-UCM-UC3M}
\curraddr{}
\email{luis.guijarro@uam.es}
\thanks{The first two authors were supported by research grants  MTM2011-22612 from the Ministerio de Ciencia e Innovaci\'on (MCINN) and MINECO: ICMAT Severo Ochoa project SEV-2011-0087
}

\author{Gerard Walschap}
\address{Department of Mathematics, University of Oklahoma.}
\curraddr{}
\email{gerard@ou.edu}
\thanks{}
\title{Twisted submersions in nonnegative sectional curvature}

\subjclass{53C20} 

\begin{abstract} In \cite{Wil}, B. Wilking introduced \emph{the dual foliation} associated to a metric foliation in a Riemannian manifold with nonnegative sectional curvature, and proved that when the curvature is strictly positive, the dual foliation contains a single leaf, so that any two points in the ambient space can be joined by a horizontal curve. We show that  the same phenomenon often occurs for  Riemannian submersions from nonnegatively curved spaces even without the strict positive curvature assumption, and irrespective of the particular metric. 
\end{abstract}

\maketitle

\section{Introduction and statements of results}

Recall that in a Riemannian manifold $M$   with a metric foliation $\calF$, the \emph{horizontal curves} are those orthogonal to the leaves of $\calF$ at every point. Given any point $p\in M$, consider the subset of $M$ that can be reached by horizontal curves emanating from $p$. In general, these subsets have no particularly interesting structure. However, in the presence of nonnegative sectional curvature  (and assuming completeness of leaves), B. Wilking showed in \cite{Wil} that they form a singular metric foliation, called the \emph{dual foliation to $\calF$},   which we hereafter denote by $\dual$.
He also showed that, when the metric is positively curved, the dual foliation contains a single leaf, and hence any two points are connected by a horizontal curve.

The simplest metric foliations appear as the collection of fibers of Riemannian submersions; it is therefore natural  to study the properties of  the dual foliation in this setting. If the total space of the submersion is compact, Theorem 3 in \cite{Wil} guarantees intrinsic completeness of the dual leaves, and hence $\dual$ is a legitimate metric foliation.   Moreover, Riemannian submersions also preserve  nonnegative sectional curvature, and provide most of the know examples of such spaces. 

The aim of this note is to show that relatively mild restrictions on a submersion often results in triviality of the dual foliation. Before describing  this in more detail, it is convenient to introduce some terminology:

\begin{defi}
A Riemannian submersion $\pi:M\to B$ is said to be \emph{twisted} if the  foliation dual to the fibers of $\pi$ contains only one leaf.
\end{defi}

The term \emph{twisted} is meant to contrast  with the case of Riemannian products $F\times B\to B$, where the dual foliations correspond to the submanifolds $\{p\}\times B$, $p\in F$. For an arbitrary Riemannian submersion $M\to B$, any closed curve in $B$ beginning and ending at $b\in B$ induces a diffeomorphism of the fiber $F= \pi^{-1}(b)$ over $b$ by lifting the curve horizontally to points of $F$. The collection of all these diffeomorphisms forms a group called the \emph{holonomy group} of $\pi$ at $b$. In the case of a Riemannian product $F\times B\to B$, that group is trivial. In general, the holonomy group is not a Lie group. However,  one large class of Riemannian submersions are the so-called \emph{homogeneous} ones: if $M$ is a Riemannian manifold, and $G$ a compact Lie group acting by isometries on $M$ with principal orbits, then the space $B=G\backslash M$ of orbits inherits in a natural way a Riemannian metric for which the quotient map $\pi:M\to B$ is a Riemannian submersion. In this case, $\pi$ is actually a fiber bundle, the holonomy group at any point is a Lie group, and is, in fact, the structure group of the bundle. Thus, if the dual foliation consists of a single leaf, then the structure group acts transitively on the fiber, and the bundle may be thought of as being  in essence twisted.

Among the several further reasons why twisted Riemannian submersions are interesting, two in particular stand out: first of all,  as mentioned earlier, they automatically occur in positive curvature, and   could help us understand topological similarities between positive and nonegative sectional curvature; second, they provide  a rich extra structure on the space,  such as the subriemannian distance on $M$ obtained by considering  the infimum of the length of horizontal curves between two points. 

In this paper, we provide several sources of twisting for Riemannian submersions. Recall that even non homogenous  Riemannian submersions are topological fibrations, and twisting can already be built into the structure of the fibration;  this is studied here in  terms of homotopy, using the boundary operator of the long exact homotopy sequence of the fibration, and in cohomological terms through the use of transgressive elements. Both are described in section \ref{topology}. Next, we  consider cases where  a submersion is either twisted or else the metric is rigidly constrained; this is illustrated for submersions of the form $\pi:M\times S^2\to M$ for arbitrary metrics on the product. Section \ref{torussection} deals with the special case of principal torus bundles, showing that the corresponding quotient map will be twisted whenever the total space is simply connected. Finally, we conclude with some  local conditions on the metric that guarantee the submersion is twisted. 

The authors are indebted to the referee for pointing out a mistake in the first version of Theorem 3, and to Kris Tapp for pointing out a mistake in subsection \ref{nonsphericalfiber} of the first version of this paper, and for telling us about reference \cite{KM}.

All the manifolds appearing in this paper are considered connected.

\section{Twist induced by the topology}\label{topology}

\subsection{The boundary operator}

R. Hermann showed in \cite{Her}  that  for complete $M$, Riemannian submersions $\pi:M\to B$ are locally trivial fiber bundles; i.e., any point $b\in B$ has a neighborhood $V$ such that $\pi^{-1}(V)$ is diffeomorphic with $U\times F$, where $F=\pi^{-1}(b)$. In particular, they are   fibrations.
Under a nonnegative curvature bound,  the dual foliation associated to the fibers of the submersion can be trivial already because the fibration itself is twisted. One way to detect this is by means of the boundary operator of the fibration:
Specifically, denote by $F$ the fiber of $\pi$ over some base point $b_0\in B$ and consider the long exact homotopy sequence 
$$
\dots \overset{\boundary}\longrightarrow \pi_q(F)\overset{i_*}\longrightarrow \pi_q(M)\overset{\pi_*}\longrightarrow\pi_q(B)\overset{\boundary}\longrightarrow\dots
$$
of the fibration $F\to M\to B$.

\begin{thm}\label{boundary}
Let $\pi:M\to B$ denote a Riemannian submersion from a complete manifold $M$ with nonnegative sectional curvature . If $\pi$ is not twisted, then for any $q$ and any $[a]\in\pi_q(B)$, the homotopy class $\boundary[a]$ can be represented by a map $f:S^{q-1}\to F$ that is not surjective. 
\end{thm}

\begin{proof}
Denote by $u^i:\R^{q+1}\to\R$ the projection that assigns to a point its $ i$-th coordinate. A homotopy class $[a]$ in $\pi_q(B, b)$ is determined by a map $g:(S^{q},e_1)\to (B,b)$, where 
\[
e_1=(1,0,\dots,0)\in S^{q-1}=\{p\in S^q\mid u^{q+1}(p)=0\}
\]
lies in the equator $S^{q-1}$. For any $p\ne e_1$ in this equator, there exists a unique plane parallel to the $u^{q+1}$-axis that contains $e_1$ and $p$. It intersects the sphere in a circle $c_p:[0,1]\to S^q$, $c_p(0)=c_p(1) =e_1$. Fix some point $m$ in $\pi^{-1}(b)$, and denote by $\tilde c_p:[0,1]\to M$ the horizontal lift starting at $m$ of $g\circ c_p$. Then $\boundary[a]$ is represented by $f:(S^{q-1}, e_1)\to( F,m)$, with $f(p)= \tilde c_p(1)$ for $p\ne e_1$, and $f(e_1)=m$. Since the image of $f$ consists of points that are reachable from $m$ by means of horizontal curves, it cannot be onto $F$ (since otherwise every point of $M$ could be reached in this way).
\end{proof}
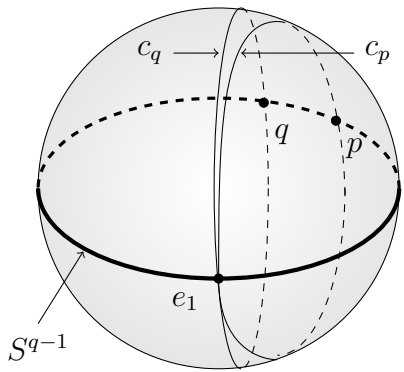
\begin{center}
\begin{tikzpicture}[scale=.6]
\shadedraw[inner color=white, outer color= gray!20!white] (5,5) circle (4cm);	
\draw [ ultra thick](1,5) arc (180:360: 4cm and 2 cm);
\draw [ very thick, dashed](9,5) arc (0:180: 4cm and 2 cm);
\filldraw (5,3) circle (1mm);
\draw (6.3,8.7) .. controls (5,8.7) and (5,6) ..  (5,3);
\draw (5,3) .. controls (5,1.5) and (5.9,1.2) ..  (6.3,1.2);
\draw [dashed] (6.3,8.7) arc (90:0:1.5cm and 4cm);
\draw (5.5,9) arc (90:270:0.6cm and 4cm);
\draw [dashed] (5.5,9) arc (90:0:0.6cm and 4cm);
\draw [dashed] (6.1,5) arc (0:-90:0.6cm and 4cm);
\draw [dashed](7.8,5) arc (0:-90:1.5cm and 3.7cm);
\filldraw (7.6,6.5) circle (1mm);
\filldraw (6,6.9) circle (1mm);
\draw (4.8,2.5) node[anchor= east]{$e_1$};
\draw (7.6,5.9) node[anchor= west]{$p$};
\draw [->] (1,2)--(2,3.6);
\draw (1,2) node[anchor=north]{$S^{q-1}$};
\draw[<-](5.5,8)--(8,8);
\draw (8,8) node[anchor=west]{$c_p$};
\draw (6, 6.2) node[anchor=west]{$q$};
\draw (4,8) node[anchor=east]{$c_q$};
\draw[->](4,8)--(5,8);
\end{tikzpicture}
\end{center}

\subsubsection{Applications to submersions with spherical fibers}
When $F$ is homotopy-equivalent to a sphere, the conclusion of Theorem \ref{boundary} can be refined:

\begin{cor}\label{fiber_sphere}
Let $\pi:M\to B$ a Riemannian submersion with fiber a homotopy sphere $S^k$.  If $\pi$ is not twisted, then for any $q\geq 2$ we have short exact sequences
$$
0 \overset{\boundary}\longrightarrow \pi_q(F)\overset{i_*}\longrightarrow \pi_q(M)\overset{\pi_*}\longrightarrow\pi_q(B)\overset{\boundary}\longrightarrow 0
$$
that split. 
\end{cor}

The proof is immediate, since in the above case the image of $[a]\in\pi_q(B)$ by the boundary map can be nonzero only if every map representing $\boundary[a]$ is onto. 

As a consequence, it follows that most of the results in \cite{GW2}, section 2, apply to this context when suitably rewritten. Specifically, we have the following:
\begin{exa}
The following Riemannian submersions are twisted for \emph{any} nonnegatively curved metric  on the total space:
\begin{enumerate}
\item The Hopf fibrations $\pi_1:S^7\to S^4$ and $\pi_2:S^{15}\to S^8$;
\item the unit sphere bundles of $T\CP^n$ and $T\HP^n$ for even $n$;
\item the unit sphere bundle of the canonical bundle of the Grassmannian of oriented $k$-planes in $\R^n$ (when $n$ is large enough).
\end{enumerate}
The submersions in the last two cases are the restrictions of the canonical bundle projections.

\end{exa}

\subsection{A twisted submersion with nonspherical fiber}\label{nonsphericalfiber}

Twisted submersions also appear naturally in fibrations  whose fiber is not a homotopy sphere. We exhibit one  where the fiber is a complex projective space. In order to do this, recall that
$S^6$ admits an almost complex structure $J:TS^6\to TS^6$; this allows us to consider the space $M$  of complex lines in the tangent bundle of $S^6$. It is easy to see that $M$ is a fiber bundle over $S^6$ with fiber $\CP^2$. 
\footnote{$M$ appears in \cite{KM}, theorem 1.3,  where it has been proved that it admits a quasipositively curved metric.}
Next, we construct a nonnegative curvature metric on $M$ for which the projection $\pi:M\to S^6$ is a Riemannian submersion: if $G_2\times S^5 $ is endowed  with the  product of the standard biinvariant metric on $G_2$ and the round metric on the sphere, then there is a free diagonal action by isometries of $SU(3)$ on this product  given by 
\[g(g_1, v)=(g_1\cdot g^{-1}, g( v)), \qquad g\in SU(3)\subset G_2, \,g_1\in G_2, \,v\in S^5.
\]
The quotient $P=G_2\times_{SU(3)} S^5$ therefore inherits a unique metric for which the projection $G_2\times S^5\to P$ is a Riemannian submersion. This metric has  nonnegative curvature by O'Neill's formula. $P$ is the total space of the unit tangent bundle of $S^6$, and the almost complex structure on $S^6$ induces an isometric $S^1$-action on $P$ via
$$
z([g_1,v])=[g_1, \cos{t} v +\sin{t} Jv],\qquad z=e^{it}\in S^1,
$$ 
with $[g_1,v]$ denoting the image of $(g_1,v)$ by the projection $G_2\times S^5\to P$.
Since the $S^1$-orbit of a unit vector $v\in T_pS^6$ is the unit circle $\{\cos{t} v +\sin{t} Jv\mid 0\le t\le 2\pi\}$ through $v$ and $Jv$,
the quotient of this action is $M$ and the induced metric is again nonnegatively curved by O'Neill's formula. Denote by $p:P\to M$ the quotient map.

Consider now the boundary maps 
\begin{eqnarray}
\partial_1: \pi_6(S^6)\to\pi_5(S^5) \\
\partial_2: \pi_6(S^6)\to\pi_5(\CP^2)
\end{eqnarray}
 in the long exact homotopy sequences  corresponding to the fibrations $S^5\to P\to S^6$ and
$\CP^2\to M\to S^6$. 

\begin{lem}
If $h:S^5\to S^5$ is the map representing $\partial_1[\Id]$, where $\Id:S^6\to S^6$ is the identity, then $p\circ h$ is the map representing 
$\boundary_2[\Id]$.
\end{lem}

\begin{proof}
Denote by $\pi_P:P\to S^6$ the unit tangent bundle projection. Since the map $p:P\to M$ is fiber preserving, it induces a homomorphism between the two homotopy sequences, such that the diagram
\[
\begin{CD}
\cdots \pi_6(S^5) @>i_{*}>> \pi_6(P) @>\pi_{P*}>> \pi_6(S^6) @>\partial_1>>\pi_5(S^5) \cdots\\
 @VVp_*V                        @VVp_*V                         @VV\text{id}V               @VVp_*V  \\
 \cdots \pi_6(\CP^2) @>>i_*> \pi_6(M) @>>\pi_*> \pi_6(S^6) @>>\partial_2> \pi_5(\CP^2)\cdots
\end{CD}
\]
commutes, see \cite{St} 17.5. The assertion then follows from commutativity of the last square in the diagram.
\end{proof}

Suppose now that $M$ has an arbitrary Riemannian metric with nonnegative curvature such that $\pi:M\to S^6$ is a Riemannian submersion for this metric, and furthermore that the corresponding dual foliation has more than one leaf. Then the map  $p\circ h$ from above would not be onto by Theorem \ref{boundary}. Since $\CP^2$ has a cell structure of the form $e_4 \cup\CP^1$, the map $p\circ h$ deforms in $\CP^2\setminus\{\text{point}\}$ to a map $h':S^5\to \CP^1\simeq S^2$. But $\pi_5(S^2)$ is of finite order, so that $p\circ h:S^5\to \CP^2$ represents an element of finite order in $\pi_5(\CP^2)\simeq\pi_5(S^5)\simeq \Z$, and hence vanishes. 

This, however, is not the case: it is well known---see for example \cite{St}  23.4---that for even-dimensional spheres $S^{2k}$, the boundary map $\partial_1:\pi_{2k}(S^{2k})\to \pi_{2k-1}(S^{2k-1})$ in the homotopy sequence of the unit tangent bundle maps a generator of the first group to two times a generator of the second one. Since $p_*:\pi_5(S^5)\to \pi_5(\CP^2)$ is an isomorphism, the composition $\partial_2=p_*\circ\partial_1$ also has that property.

\subsection{The transgression} The homotopy conditions appearing in Theorem \ref{boundary} can be reformulated in cohomological terms using \emph{transgressive elements}; we begin by recalling some of the main definitions from \cite{BT}.

Let $\pi:M\to B$ a Riemannian submersion with fiber $F$. An element $\omega\in H^q(F)$ is said to be \emph{transgressive} if  it equals the restriction of a global $q$-form $\psi$ on $M$ such that $d\psi=\pi^*\tau$ for some $(q+1)$-form $\tau$ on $B$. Since $\pi^*$ is injective, $\tau$ is a closed form.  The correspondence that assigns $\tau$ to $\omega$ is not well defined as such; its domain consists of $\omega$'s in a certain subgroup of $H^q(F)$, and the $\tau$'s are defined modulo some subgroup of $H^{q+1}(B)$. When this is done, the map obtained is a homomorphism $T$ called \emph{the transgression}.
The reader interested in further details is invited to consult  the above reference and   \cite{MC}. 

We will also need the following definition used in \cite{GSW}:
\begin{defi}
     Let $X$ be a CW-complex and $a\in H^k(X,\Z)$ a cohomology
     class.  $a$ is said to  be \emph{spherical} if there is a map $f\colon S^k\to X$
     such that $f^*a \ne 0\in H^k(S^k,\Z)$. Otherwise, $a$ is
     said to be \emph{not spherical}.
\end{defi}

Notice that a spherical element has infinite order.  Moreover, such a
cohomology class is never a 
cup product of lower dimensional classes. 

\begin{thm}\label{transgress}
Let $\omega$ be a transgressive cohomology class in $H^k(F)$, where $k=\dim F$. If $\pi:M\to B$ is a non twisted Riemannian submersion, then any cohomology class representing $T(\omega)$ is not spherical.
\end{thm}
\begin{proof}
Assume otherwise; then there is a map $f:S^{k+1}\to B$ such that 
$f^*\tau\neq 0$, where $\tau$ represents $T(\omega)$. Consider the cohomology class $[a]$ of $B$ given by $f_*[S^{k+1}]$.
%
$[a]$ can be represented by a smooth map $g:D^{k+1}\to B$ mapping $\partial D^{k+1}$ to a fixed point of $B$. As in the proof of Theorem \ref{boundary}, lift $g$ horizontally to a map $G:D^{k+1}\to M$, and denote
by $g:S^k\to F$ its restriction to the boundary of $D^{k+1}$. If the submersion $\pi$ is not twisted, $g$ cannot be onto.  Then
\begin{equation}
\begin{split}
0\neq \ip{\tau}{[a]}&=\int_{[a]}\,\tau = \int_{G(D^{k+1})}\, \pi^*\tau =  \int_{G(D^{k+1})}\, d\psi 
 =\int_{g(S^k)}\, \psi \\
&= \int_{g(S^k)}\, \omega. 
\end{split}
\end{equation}
However since $g(S^k)$ is not the whole fiber $F$, the last integral vanishes.
\end{proof}

For the applications, it is convenient to recall that 
the transgression map $T$ appears in the Leray-Serre exact sequence of the fibration $F\to M\to B$ as the differential 
$d_n:E_n^{0,n-1}\to E_n^{n,0}$, where $E_n^{0,n-1}$ is contained in $H^{n-1}(F)$.   

\begin{exa}
Consider the fibration $SO(3)\to \RP^7\to S^4$. The transgression map is nontrivial because the differential $d_4$ above does not vanish: if it were trivial, then  the exact sequence would  imply that $H^4(\RP^7)\neq 0$. Since every nonzero class in $H^4(S^4)$ is spherical, Theorem \ref{transgress} implies that $\pi:\RP^7\to S^4$ is twisted. Once again, this fibration has non-spherical fiber.  
\end{exa}

\section{Twisting in products}
If $M_i$, $i=1,2$, are Riemannian manifolds with nonnegative curvature, then the product metric on $M_1\times M_2$ also has nonnegative curvature. One interesting question is whether there exist other metrics of nonnegative curvature on the product. We investigate this here in a specific context:
\begin{thm}
Let $M$ be a compact simply connected manifold. Assume there is a nonnegatively curved metric on $N=M\times S^2$ such that there is a Riemannian submersion $\pi:N\to M$ with totally geodesic fibers. Then either:
\begin{enumerate}
\item $N$ is isometric to the Riemannian product $M\times S^2$ (for some nonnegatively curved $S^2$), or
\item $\pi$ is twisted.
\end{enumerate} 
\end{thm}
\begin{proof} 
If there is more than one leaf in the dual foliation, then these leaves must have codimension one or two in $N$ because of Theorem 3 in \cite{Wil}. Notice that not every leaf has codimension one, since the intersection of the leaves with a fiber of $\pi$ would result in a line field on $S^2$. Thus, there must be some dual leaf of codimension 2; pick one such, and call it  $P$; since $P$ is horizontal, the restriction of $\pi$ to $P$ is a covering of $M$, and $\pi$ maps $P$  isometrically onto the base $M$. 

The inverse of the restriction  $\pi_{|P}$  of  $\pi$ to $P$ yields a section $s:M\to N$ of the type studied in \cite{BG}, section 6.  Consider the normal bundle $\nu(P)$ of $P$ in $N$, and let $\epsilon>0$ be small enough so that the restriction
\[\exp:\nu^\epsilon(P)=\{u\in\nu(P)\mid |u|<\epsilon\} \longrightarrow N
\]
of the exponential map is a diffeomorphism onto its image. Endow $\nu^\epsilon(P)$ with the connection metric for which $\exp$ becomes an isometry. Then the bundle projection $\nu^\epsilon(P)\to P$ is a Riemannian submersion with totally geodesic fibers. Since the bundle is trivial, the arguments used in the proof of Theorem 1.5 in \cite{Wa0} carry over verbatim to imply that $\nu^\epsilon(P)$ splits locally isometrically over $P$.
 In particular, the horizontal distribution of $\pi$ is integrable in an open neighborhood of $P$. By the definition of  dual foliation, each leaf that intersects such a neighborhood has codimension 2 and is entirely contained in it. Thus, the set of points where the O'Neill tensor is zero is both open and closed in $N$,  and $A$ vanishes everywhere. By \cite{Wa} (see also \cite{GP}), $N$ splits metrically as a product. 
\end{proof}

\section{Twisting in principal torus bundles}\label{torussection}

Many of the examples of Riemannian submersions appear as quotient maps of free isometric actions. In this section, we investigate twisting for torus actions.

\begin{thm} \label{torus}
Consider a free isometric $T^k$ action on  a Riemannian manifold $M$  with nonnegative sectional curvature. If $M$ is simply connected, then the quotient submersion $\pi:M\to M/T^k$ is twisted.
\end{thm}

\begin{proof} 

Consider the sequence of homogeneous metric fibrations
\[M\overset{p_1}{\to}M_1\overset{p_2}{\to}\dots M_{k-1}\overset{p_k}{\to}M_k=M/T^k
\]
each of which has a circle as fiber. It is enough to show that each $p_i:M_{i-1}\to M_i$ has only one dual leaf. So, assume to the contrary that some $p_i$ has more than one leaf. Since dual leaves are isometric via the $S^1$ action, they form a regular codimension 1 metric foliation; thus, the $A$-tensor of this foliation vanishes, and since the ambient space has nonnegative curvature, it splits locally isometrically as a product with one factor tangent to the $S^1$ action. Since $S^1$ is closed, this splitting is global; i.e., $M_{i-1} = N\times S^1$ isometrically. We will get a contradiction once we show that this splitting lifts to a splitting of $M$ via $p_1$,..., $p_{i-1}$. We may inductively assume that $i= 1$; i.e., we have a homogeneous submersion $\pi:S^1\to M\to N\times S^1$ such that if $T$ is the vertical Killing field on $M$ generating the fibers of $\pi$, and $\bar X$ is the Killing field on the base tangent to the $S^1$ factor, then the basic lift $X$ of $\bar X$ is a Killing field on $M$ that commutes with $T$ (since the $S^1$ action on $N\times S^1$ is induced by an isometric torus action on $M$). The claim will follow once it is  shown that $X$ is a parallel vector field on $M$. 
Now if $Y$ is basic, then $(\nabla_YX)^h=0$ since the projected fields on the base have that property, and $(\nabla_YX)^v=A_YX=0$ since the projected fields span a plane of zero curvature. It remains to prove that $\nabla_TX=0$. But $(\nabla_TX)^h=0$ because
\[\langle \nabla_TX, Y\rangle=-\langle T, A_XY\rangle=0,
\]
and $(\nabla_TX)^v=0$ because $\langle\nabla_TX,T\rangle =0$ since $X$ is Killing (so that the operator $u\mapsto\nabla_u X$ is skew-adjoint).
\end{proof}

This has several simple applications: consider for example $M=SU(n)$, some (not necessarily maximal) torus $T^k$ contained in $ SU(n)$, and a Riemannian metric on $SU(n)$ invariant for $T^k$. Theorem \ref{torus} then implies that the quotient map $SU(n)\to SU(n)/T^k$  is twisted.

\section{Twisting due to the O'Neill tensor}
The O'Neill $A$ tensor of a Riemannian submersion measures the extent to which the horizontal distribution, as a subbundle of $TM$, fails to be integrable. It is often responsible for the twisting of the submersion, as illustrated in the next lemma. We first recall some terminology:

A Riemannian submersion $\pi:M\to B$ induces an orthogonal splitting $TM=\calH\oplus\calV$ of the tangent bundle of $M$, with $\calH$ denoting the horizontal, and $\calV=\ker\pi_*$ the vertical, distribution. There is a corresponding decomposition
\[z=z^\ho+z^\ve\in\calH\oplus\calV,\qquad z\in TM
\]
at the vector level. There are two tensor fields that measure the complexity of $\pi$. One
 is the $A$ tensor of O'Neill,  given at each point $p\in M$ by a map
$$
A:\calH_p\times \calH_p \to \calV_p, \qquad A_xy=[X,Y]^\ve(p),
$$
where  $X$, $Y$ are horizontal extensions of $x$,  $y$ to a neighbourhood of $p$. The antisymmetry of $A$, $A_xy=-A_yx$, means that we can consider $A$ as a linear map
$$
\tildeA_p:\Lambda^2(\calH_p)\to \calV_p, \qquad \tildeA_p(\sum_i x_i\w y_i)=\sum_i A_{x_i}y_i.
$$
The other is the second fundamental tensor of the fibers, given at $p\in M$ by
\[S:\calH_p\times\calV_p\to\calV_p, \qquad S_xu=-\nabla_U^\ve X,
\]
with $X$ as above, and $U$ a vertical field extending $u$ in a neighborhood of $p$.

\begin{lem}\label{Atensor}
Let $\pi:M\to B$ be a Riemannian submersion with nonnegative curvature;
if for some point $p\in M$, $\tildeA_p$ is onto, then the dual leaf through $p$ is open.
\end{lem}
\begin{proof}
Suppose that the leaf $\calF_p$ through $p$ has  dimension less than that of $M$. By Proposition 6.1 in \cite{Wil}, the parallel transport of each normal direction to $\calF_p$ along a horizontal geodesic of $\pi$ generates a Jacobi field $J$. It is immediate to see that $J$ is in fact a holonomy Jacobi field for $\pi$ and hence  satisfies the equation
$$
J'(t)= -A^*_{\g'(t)}J + S_{\g'}J,
$$   
where  $\g$  denotes an arbitrary horizontal geodesic with $\g(0)=p$, and $A^*$ is the adjoint of the O'Neill tensor (see \cite{GrW}). Thus, at $t=0$, $A^*_{\g'(0)}J(0)=0$, and therefore $J(0)$ is orthogonal to the image of $\tildeA_p$, which is impossible. 
\end{proof}

The similarity of this result with Theorem C in \cite{GW1}, where a similar theorem was proved for the Sharafutdinov submersion, is somewhat intriguing. 

\begin{cor}\label{principal}
Let $M$ be a principal bundle over $B$ with an invariant metric of nonnegative curvature;
if for some point $p\in M$, $\tildeA_p$ is onto, then the submersion is twisted.
\end{cor}
\begin{proof}
Observe that  the isometric action preserves the dual leaves, and hence they will be all of the same dimension; 
if $\pi:M\to B$ were not twisted, then the leaf $\calF$ through $p$ would have  dimension less than that of $M$, and could not be open.
\end{proof}

\begin{cor}
A principal circle bundle $M$ with nonnegative curvature  is either twisted or else it splits locally as the projection of a metric product onto one of the factors.
\end{cor}
\begin{proof}
This is immediate from Lemma \ref{Atensor}, for if the submersion is not twisted, then the $A$ tensor must vanish everywhere, and the splitting  follows from results in \cite{Wa}. 
\end{proof}

It has been pointed out to the authors that a result similar to corollary \ref{principal} has been obtained in \cite{Shi}, although by different methods.

%
%

\end{document}